\documentclass[11pt,english,a4paper]{smfart}
\usepackage[english]{babel}
\usepackage{xcolor}
\usepackage{pbsi}
\usepackage[T1]{fontenc}
\usepackage{stmaryrd}
\usepackage{mathabx}
\usepackage[mathscr]{euscript}
\usepackage[pagebackref,colorlinks=true,linkcolor=blue,citecolor=blue,urlcolor=red, hypertexnames=true]{hyperref}
 \usepackage[abbrev,backrefs,nobysame]{amsrefs}
 \usepackage{caption}
 \usepackage{enumerate}
\usepackage{amsmath,amssymb,bm,amsfonts}
\usepackage[all]{xy}
\entrymodifiers={+!!<0pt,\fontdimen22\textfont2>}

\usepackage{mathrsfs}

\usepackage{graphicx}

\usepackage{color}

\newtheorem{theorem}{Theorem}[section]
\newtheorem{lemma}[theorem]{Lemma}

\newtheorem{corollary}[theorem]{Corollary}
\newtheorem{proposition}[theorem]{Proposition}

\renewcommand*{\thesubsection}{\thesection\alph{subsection}}

\newcommand{\flba}[2]{
\xymatrix@C15pt{#1\ar@{|->}[r]&#2}}
\newcommand{\flcourte}[2]{
\xymatrix@C12pt{#1\ar[r]&#2}}

{\theoremstyle{definition}}

{\theoremstyle{definition}\newtheorem{example}[theorem]{Example}}

\theoremstyle{definition}

\newtheorem{question}[theorem]{Question}
\newtheorem{fact}[theorem]{Fact}

\newtheorem*{remark*}{\it Remark}

\newtheorem*{question*}{\it Question}

{\theoremstyle{definition}\newtheorem{remark}[theorem]{Remark}}

\def\D{\ensuremath{\mathbb D}}

\def\T{\ensuremath{\mathbb T}}
\def\R{\ensuremath{\mathbb R}}
\def\Z{\ensuremath{\mathbb Z}}

\def\C{\ensuremath{\mathbb C}}

\def\N{\ensuremath{\mathbb N}}

\def\bth{\begin{theorem}}
\def\blm{\begin{lemma}}
\def\bpr{\begin{proposition}}
\def\bpf{\begin{proof}}
\def\epf{\end{proof}}
\def\epr{\end{proposition}}
\def\elm{\end{lemma}}
\def\eth{\end{theorem}}
\def\bco{\begin{corollary}}
\def\eco{\end{corollary}}
\def\be{\begin{enumerate}}
\def\ee{\end{enumerate}}
\def\bea{\begin{enumerate}[\rm (a)]}
\def\beun{\begin{enumerate}[\rm (1)]}
\def\bei{\begin{enumerate}[\rm (i)]}

\newcommand{\pss}[2]{\ensuremath{{\langle #1,#2\rangle}}}

\newcommand{\ba}[1]{\overline{#1}}

\newcommand{\gd}{G_{\delta }}

\newcommand{\bmx}{{\mathcal{B}}_{M}(X)}
\newcommand{\bbx}{{\mathcal{B}}_{1}(X)}

\newcommand{\wot}{\texttt{WOT}}

\newcommand{\sot}{\texttt{SOT}}
\newcommand{\sote}{\texttt{SOT}\mbox{$^{*}$}}
\newcommand{\bx}{{\mathcal B}(X)}

\def\lp{\ell_{p}}

\DeclareMathOperator{\dist}{dist}
\numberwithin{equation}{section}

\author[S. Grivaux]{Sophie Grivaux}
\address[S. Grivaux]{CNRS, Univ. Lille, UMR 8524 - Laboratoire Paul
Painlev\'e, F-59000 Lille, France}
\email{sophie.grivaux@univ-lille.fr}
\author[\'{E}. Matheron]{\'{E}tienne Matheron}
\address[\'{E}. Matheron]{Laboratoire de Math\'{e}matiques de Lens\\ Universit\'{e} d'Artois\\ Rue Jean Souvraz SP 18\\ 62307 Lens, France}
\email{etienne.matheron@univ-artois.fr}

\begin{document}

\title[Local spectral properties of typical contractions]{Local spectral properties of typical contractions on \(\lp\,\)-$\,$spaces}

 \thanks{This work was supported in part by
the project FRONT of the French
National Research Agency (grant ANR-17-CE40-0021) and by the Labex CEMPI (ANR-11-LABX-0007-01).}

\begin{abstract}  We study some local spectral properties of contraction operators on $\ell_p$, $1<p<\infty$ from a Baire category point of view, with respect to the Strong$^*$ Operator Topology. 
In particular, we show that a typical contraction on $\ell_p$ has Dunford's Property (C) but neither Bishop's Property $(\beta)$ nor the Decomposition Property $(\delta)$, and is completely indecomposable. We also obtain some results regarding the asymptotic behavior   
of orbits of typical contractions on $\lp$. 
\end{abstract}

\maketitle

\hfill\emph{To the memory of J\"org Eschmeier}

\par\bigskip
\section{Introduction}\label{Introduction}
In this note, we shall be interested in the \emph{typicality}, in the sense of Baire Category, of some natural properties of Banach space operators pertaining to \emph{local spectral theory}. 
\par\medskip
In the whole paper, the letter \(X\) will denote an infinite-dimensional complex (separable) Banach space with separable dual. Most of the time, $X$ will be in fact \(\lp:=\lp(\Z_{+})\) for some \(1<p<\infty\). %We write \(\lp=\lp(\Z_{+})\), and we denote by \((e_{j})_{j\ge 0}\) the canonical basis of \(\lp\). 
We denote by \(\bx\) the algebra of bounded operators on \(X\), and by \(\bbx\) the set of contraction operators on \(X\):
\[
\bbx=\bigl\{T\in\bx\;;\;||T||\le 1\bigr\}.
\]
There are several natural topologies on \(\bbx\) which may turn it into a Polish (\textit{i.e.} separable and completely metrizable) space, thus allowing for a study of  ``large'' sets of contractions on \(X\) in the sense of Baire Category. Some well-known such topologies  are the Weak Operator Topology (\texttt{WOT}), the Strong Operator Topology (\sot), and the 
Strong\(^{*}\) Operator Topology (\sote). In this note, we shall almost exclusively be concerned with the topology \sote, which is defined as follows: a net \((T_{i})\subseteq\bx\) converges to \(T\in\bx\) with respect to \sote\ if and only if \(T_{i}x\) tends to \(Tx\) in norm for every \(x\in X\)  and 
\(T_{i}^{*}x^{*}\) tends to \(T^{*}x^{*}\) in norm for every \(x^{*}\in X^{*}\). In other words, \(T_{i}\to T\) for \sote\ if and only if \(T_{i} \to T\) and \(T_{i}^{*}\to T^{*}\) for \sot. Under our assumptions on \(X\), it is well-known that \(\bbx\) is a Polish space when endowed with \sote. %This opens the door to Baire Category arguments, and to the study of typical properties of operators in \((\bbx,\sote)\). 
In addition to its Polishness, one pleasant feature of the topology \sote\ is that when \(X\) is reflexive, the map \(T\mapsto T^{*}\) is a homeomorphism from \(\bbx\) onto \(\mathcal{B}_{1}(X^{*})\). 

\medskip
In this paper, the word \emph{typical} will always refer to the topology \sote, unless the contrary is explicitly mentioned. Thus, we will say that a given property  (P) of operators $T\in\bbx$ is typical, or that \emph{a typical $T\in\bbx$ has Property \emph{(P)}}, if the set $\bigl\{T\in\bbx\;;\;T\ \hbox{has property (P)}\bigr\}$ is comeager in \((\bbx,\sote)\), \textit{i.e.}\ contains a dense \(G_{\delta }\) subset of 
\((\bbx,\sote)\). In other words, 
a property (P) of contractions on \(X\) is typical if the set of all \(T\in\bbx\) enjoying it is very large in the sense of Baire Category. 

\medskip Typical properties of \emph{Hilbert} space contractions with respect to the topologies \wot, \sot\ and \sote\ are investigated in depth in  \cite{Ei} and \cite{EM}. The main conclusions drawn from these works are that the generic situation is rather well understood in the case of \wot\ (a typical contraction is unitary); that it is ``trivial'' in the case of \sot\ (a typical contraction is unitarily equivalent to the backward shift with infinite multiplicity, so there is generically \emph{just one} Hilbert space contraction); and that things are more complicated for the topology \sote
. This line of thought is pursued in   
\cite{GMM1} and \cite{GMM2}. The setting of \cite{GMM1} is mostly Hilbertian, although many of the proofs work as well on \(\lp\,\)-$\,$spaces; and for that reason, emphasis is mostly put on the topology \sote.  In \cite{GMM2}, the setting is explicitly that of 
\(\lp\,\)-$\,$spaces, and the typicality of certain properties of operators pertaining to the study of the Invariant Subspace Problem is considered for the topologies \sot\ and \sote.
\par\medskip
This note may be viewed as an \textit{addendum} to  \cite{GMM2}. 
Our aim will be to determine whether certain properties of operators related to local spectral theory  are typical in 
\(\mathcal{B}_{1}(\lp)\), \(1<p<\infty\). As will be explained below, our study was in part motivated by results from \cite{Br2}, \cite{EP1}, and \cite{EP2} showing the existence of non-trivial invariant subspaces under suitable spectral assumptions.

\medskip Throughout the paper, we denote by $\D$ the open unit disk in $\C$, and by $\overline{\,\D}$ the closed unit disk. Also, we denote by $\sigma(T)$ the spectrum of an operator $T\in\bx$. Thus, we have $\sigma(T)\subseteq \overline{\,\D}$ for every $T\in\bbx$.
\section{Background from local spectral theory}\label{Section Background from local spectral theory}
In order to be as self-contained as possible, we recall in this section some basic definitions as well as some important results from local spectral theory. We follow the terminology and notations of \cite{LN}.

\subsection{Decomposable operators} A large part of local spectral theory is motivated by the study of so-called \emph{decomposable operators}. An operator \(T\in\bx\) is decomposable if for any covering \((U_{1},U_{2})\) of \(\C\) by two open subsets, there exist two closed \(T\)-$\,$invariant subspaces \(E_{1}\) and \(E_{2}\) such that \(\sigma (T_{|E_{1}})\subseteq U_{1}\), \(\sigma (T_{|E_{2}})\subseteq U_{2}\), and \(X=E_{1}+E_{2}\). For instance, it follows from the Riesz Decomposition Theorem that any operator with totally disconnected spectrum is decomposable. Also, any surjective isometry is decomposable (see \cite[Proposition 1.6.7]{LN}). Two related notions introduced and studied in \cite{AtzSod} are \emph{strong indecomposability},  and \emph{complete indecomposability}: an operator \(T\in\bx\) is {strongly indecomposable} if its spectrum \(\sigma (T)\) is not a singleton and is included in the spectrum of the restriction of \(T\) to any of its (closed) non-zero invariant subspaces; and $T$ is {completely indecomposable} if both \(T\) and \(T^{*}\) are strongly indecomposable.

\subsection{Local spectra} 
The \emph{local resolvent set} of an operator \(T\in\bx\) at a vector \(x\in X\) is the open subset \(\rho _{x}(T)\) of \(\C\) defined in the following way: a complex number \(\lambda \) belongs to \(\rho _{x}(T)\) if there exists a holomorphic function \(f:V\to X\) defined on some neighborhood \(V\) of \(\lambda \) such that 
\((T-z)f(z)=x\) for every \(z\in V\). Note that the usual resolvent set  \(\rho (T):=\C\,\setminus\, \sigma (T)\) is contained in \(\rho _{x}(T)\): for any $\lambda\in\rho(T)$, we may take $V:=\rho(T)$ and we \emph{must} take $f(z):=(T-z)^{-1}x$ for every $z\in\rho(T)$. 	The \emph{local spectrum} of \(T\) at \(x\) is the compact set \(\sigma _{x}(T)=\C\,\setminus\, \rho _{x}(T)\), which is contained in \(\sigma (T)\). Obviously \(\sigma _{x}(T)=\emptyset\) if \(x=0\); but it may happen that \(\sigma _{x}(T)=\emptyset\) for some non-zero vector \(x\).

\subsection{Eigenvector fields, and SVEP}\label{svep} Let $T\in\bx$. A \emph{holomorphic $T\,$-$\,$eigenvector field} is a holomorphic function $E:V\to X$ defined on some open set $V\subseteq\C$, such that $(T-z)E(z)=0$ for every $z\in V$. (The terminology is slightly inaccurate since $E(z)$ is allowed to be $0$.) The operator $T$ has the \emph{Single Valued Extension Property} (SVEP) if for any open set $V\subseteq\C$, the only holomorphic $T\,$-$\,$eigenvector field $E:V\to X$ is $E=0$. It is rather clear that if $T$ has SVEP then, for any $x\in X$, there is a unique holomorphic function $f:\rho_x(T)\to X$ solving the equation $(T-\lambda)f(\lambda)\equiv x$ on $\rho_x(T)$, which is called the \emph{local resolvent function for $T$ at $x$}. Moreover, it can be shown that $T$ has SVEP if and only if $\sigma_x(T)\neq\emptyset$ for every $x\neq 0$ (see \cite[Proposition 1.2.16]{LN}). Obvious examples of operators with SVEP are all operators whose set of eigenvalues has empty interior. In particular, any isometry has SVEP. On the other hand, the usual backward shift $B$ on $\ell_p$ lacks SVEP in a strong way, since it has a \emph{spanning} holomorphic eigenvector field defined in the unit disk $\D$, \mbox{\it i.e.} a holomorphic eigenvector field  $E:\D\to \ell_p$ such that $\overline{\rm span} \,\{ E(z);\; z\in\D\}=\lp$; namely, $E(z):=\sum_{j=0}^\infty z^j e_j$, where $(e_j)_{j\geq 0}$ is the canonical basis of $\ell_p$. More generally, by a classical result of Finch \cite{F} (see also \cite[Proposition 1.2.10]{LN}), any surjective but non-invertible operator lacks SVEP. At the other extreme, the usual \emph{forward} shift $S$ on $\ell_p$ has ``{maximal local spectra}'', \mbox{\it i.e.} it satisfies $\sigma_x(S)=\overline{\,\D}$ for every $x\neq 0$. 
In fact, the Remark after \cite[Theorem 1.5]{Vr} shows the following: if $T\in\mathcal B(X)$ is such that its Banach space adjoint $T^*\in\mathcal B(X^*)$ admits a spanning holomorphic eigenvector field $E$ defined on some connected open set $\Omega\subseteq\C$, then $\sigma_x(T)$ contains $\overline\Omega$ for every $x\neq 0$. 

\subsection{Local spectral radii} If \(T\in\bx\) and \(x\in X\), the \emph{local spectral radius} of \(T\) at \(x\) is the number
\[
r_{x}(T):=\limsup_{n\to\infty}\,||T^{n}x||^{1/n}.
\]
It is not difficult to see that \(r_{x}(T)\ge\sup\{|\lambda |\;;\;\lambda \in\sigma _{x}(T)\}\); and a littler harder to show that if \(T\) has SVEP, then \(r_{x}(T)=\sup\{|\lambda |\;;\;\lambda \in\sigma _{x}(T)\}\) for every \(x\neq 0\). This is the \emph{local spectral radius} formula, see \cite[Proposition 3.3.13]{LN}.

\subsection{Spectral subspaces} Let $T\in\bx$. For any subset \(F\) of \(\C\), let
\[
X_{T}(F):=\{x\in X\;;\;\sigma _{x}(T)\subseteq F\}.
\]
With this notation, we see from \ref{svep}  that the operator \(T\) has SVEP if and only if \(X_{T}(\emptyset)=\{0\}\); and it turns out that this holds if and only if \(X_{T}(\emptyset)\) is closed in \(X\). Also, it can be shown that \(X_{T}(F)\) is always a \(T\,\)-$\,$hyperinvariant linear subspace of \(X\). See \cite[Proposition 1.2.16]{LN}. Moreover, if $T$ has SVEP and if $F$ is a closed set such that $X_T(F)$ is closed, then $\sigma\bigl(T_{| X_T(F)}\bigr)\subseteq \sigma(T)\cap F$ (see \cite[Proposition 1.2.20]{LN}). 

\medskip
The subspaces $X_T(F)$ may be called \emph{local spectral subspaces} for $T$, since $x\in X_T(F)$ means that for any $\lambda\in\C\setminus F$, there is a holomorphic solution of the equation $(T-z)f(z)\equiv x$ defined in a neighborhood of $\lambda$. When $F$ is closed, 
one can also define a \emph{global} spectral subspace, denoted by $\mathcal X_T(F)$: a vector $x$ belongs to $\mathcal X_T(F)$ if there is a holomorphic solution of $(T-z)f(z)\equiv x$ globally defined on $\C\setminus F$. Obviously $\mathcal X_T(F)\subseteq X_T(F)$; and if $T$ has SVEP then $\mathcal X_T(F)=X_T(F)$ for any closed set $F\subseteq \C$.

\subsection{Property (C)} 
An operator $T\in\bx$ is said to have \emph{Dunford's Property $(\emph{C})$} if $X_T(F)$ is closed in $X$ for every closed set $F\subseteq \C$. By what has been said just after the definition of $X_T(F)$, Property (C) implies SVEP. Moreover, it can be shown that decomposable operators have Property (C), and in fact that an operator $T\in\bx$ is decomposable if and only if (i) $T$ has Property (C) and (ii) for every open covering $(U_1, U_2)$ of $\C$, it holds that $X=X_T(\overline{\,U_1})+X_T(\overline{\,U_2})$. See \cite[Theorem 1.2.23]{LN}.

\subsection{Property $(\delta)$} The global spectral subspaces $\mathcal X_T(F)$ are involved in the definition of the important \emph{Decomposition Property} \((\delta )\): an operator \(T\in\bx\) has Property \((\delta )\) if for any open covering \((U_{1},U_{2})\) of \(\C\), we have \(\mathcal X_{T}(\overline{\,U_{1}})+\mathcal X _{T}(\overline{\,U_{2}})=X\). Every decomposable operator has property \((\delta )\); and more precisely, an operator is decomposable if and only if it has both properties (C) and \((\delta )\). See \cite[Proposition 1.2.29]{LN}.

\subsection{Property $(\beta)$} An operator $T\in\bx$ has \emph{Bishop's Property $(\beta)$} if the following holds true: for any open set $V\subseteq \C$ and any sequence of holomorphic functions $\phi_n:V\to X$, if $(T-\lambda)\phi_n(\lambda)\to 0$ uniformly on compact subsets of $V$, then $\phi_n(\lambda)\to 0$ uniformly on compact sets. Obviously, Property $(\beta)$ implies SVEP; and in fact, Property $(\beta)$ implies Property (C) (see \cite[Proposition 1.2.19]{LN}). That the converse is not true is a classical result due to Miller and Miller \cite{MiMi} (see also \cite[Theorem 1.6.16]{LN}). On the other hand, decomposable operators have Property $(\beta)$ (see \cite[Theorem 1.2.7]{LN}); and the link between decomposability and Property $(\beta)$ is actually much deeper: by a famous result of Albrecht and Eschmeier \cite{AE} (see also \cite[Theorem 2.4.4]{LN}), an operator $T\in\bx$ has Property $(\beta)$ if and only if it is similar to the restriction of some {decomposable} operator to one of its invariant subspaces.

Since property \((\beta )\) implies property (C), and since \(T\) is decomposable if and only if it has both properties (C) and \((\delta )\), it follows that \(T\) is decomposable if and only if it has both properties \((\beta )\) and \((\delta )\).
\par\medskip
Also, since invertible isometries are decomposable and since every (not necessarily invertible) isometry is the restriction of some invertible isometry on a larger Banach space by a classical result of Douglas \cite{D} (see also \cite[Proposition 1.6.6]{LN}), we see that every isometry has Property $(\beta)$, and hence Property (C). (Incidentally, this shows that an operator $T$ may have Property (C) whereas $T^*$ does not and even lacks SVEP: consider the usual forward shift $S$ on $\ell_p$, for which $X_T(F)=X$ if $F\supseteq\overline{\,\D}$ and $X_T(F)=\{ 0\}$  otherwise.) 

\medskip Finally, there is a quite remarkable duality between properties $(\beta)$ and $(\delta)$, due to Albrecht-Eschmeier \cite{AE} (see also \cite[Theorem 2.5.18]{LN}): an operator $T$ has Property $(\beta)$ if and only if $T^*$ has Property $(\delta)$, and $T$ has Property $(\delta)$ if and only if $T^*$ has Property $(\beta)$. In particular, $T$ is decomposable if and only if $T^*$ is decomposable.

\subsection{Invariant subspaces} Local spectral theory provides sophisticated and extremely efficient tools for extending to the Banach space setting some Hilbertian invariant subspaces results concerning operators with \emph{thick spectrum}. Recall that a compact subset \(K\) of \(\C\) is said to be \emph{thick} if there exists a non-empty bounded open set \(U\subseteq \C\) in which \(K\) is \emph{dominating}, which means that $\sup_{z\in U} \vert f(z)\vert= \sup_{z\in U\cap K} \vert f(z)\vert$ for every bounded holomorphic function $f$  on $U$.
\par\medskip
A well-known result of Brown \cite{Br2} states that any hyponormal operator with thick spectrum on a Hilbert space admits a non-trivial invariant closed subspace. This result admits several extensions to the Banach space setting, perhaps the ultimate one being due to Eschmeier and Prunaru \cite{EP2}. It involves the notion of \emph{localizable spectrum} of an operator. If \(T\in\bx\), the {localizable spectrum} \(\sigma _{loc}(T)\) of \(T\) is the closed subset of \(\C\) defined as
\[
\sigma _{loc}(T):=\bigl\{\lambda \in\C\;;\;\mathcal X_{T}(\ba{V})\neq\{0\}\ \textrm{for any neighborhood}\ V\ \textrm{of}\ \lambda \bigr\}.
\]

\noindent
It is not difficult to see, using Liouville's Theorem, that \(\sigma _{loc}(T)\subseteq \sigma (T) \); and it is shown in \cite[Theorem 2]{MiMiN} that equality holds if $T$ has Property $(\delta)$. Regarding the existence of invariant subspaces, the main result of \cite{EP2} runs as follows: {If \(T\in\bx\) is such that either \(\sigma _{loc}(T)\) or \(\sigma _{loc}(T^{*})\) is thick, then \(T\) has a non-trivial closed invariant subspace.} Since \(\sigma _{loc}(T)=\sigma (T)\) if $T$ has Property $(\delta)$, it follows that any operator with Property \((\delta )\) or Property \((\beta )\) and with thick spectrum has a non-trivial invariant subspace (\cite{EP1}, see also \cite[Theorem 2.6.12]{LN}). As any hyponormal operator on a Hilbert space has property \((\beta )\), this extends indeed the Hilbertian result from \cite{Br2}.
\par\medskip
This result from \cite{EP2} was the initial motivation for this work. Indeed, \cite{GMM2} left open the following basic question: does a typical operator on \(\lp\) for \(1<p\neq 2<\infty\) has a non-trivial invariant subspace? As a typical operator \(T\) on \(\lp\) satisfies \(\sigma (T)=\ba{\,\D}\) by \cite[Proposition 7.2]{GMM2}, and hence has thick spectrum, it is natural, in view of \cite{EP2} to wonder whether a typical \(T\in\mathcal B_1(\lp)\) has Property \((\delta) \) or Property \((\beta )\), or has thick localizable spectrum. We will unfortunately see that none of this happens.  
Thus, it seems that methods from local spectral theory are not likely to help much for showing that a typical contraction on $\lp$, $p\neq 2$, has a non-trivial invariant subspace. For \(p=2\), it follows from the Brown-Chevreau-Pearcy Theorem \cite{BCP}  and from the fact that a typical \(T\) is such that \(\sigma (T)=\overline{\,\D}\), that a typical \(T\in \mathcal{B}_{1}(\ell_2)\) does have a non-trivial invariant subspace.

\subsection{Results} The rest of the paper is organized as follows. In Section \ref{Section 3}, we show that typical operators on \(\lp\) have {maximal} local  spectra, and we draw several consequences from this. Most notably, we show that a typical $T\in\bbx$ has Property (C) but has neither Property $(\delta)$ nor Property $(\beta)$, is completely indecomposable and has empty localizable spectrum. In Section \ref{Section 4}, we present some results regarding the asymptotic behavior of orbits of typical operators on \(\lp\). The general idea is that, for a typical $T\in\mathcal B_1(\ell_p)$, no non-zero orbit can be ``too small'', yet most orbits are ``partly small''. One consequence of these results is that for any $M>1$, a typical $T\in\mathcal B_M(\ell_2)$ satisfies a strong form of \emph{distributional chaos}. In Section \ref{Section 5}, we deviate a little bit from our main topic: we state a simple condition ensuring that an operator lacks SVEP in a strong way, and we use this to give examples of operators admitting a mixing Gaussian measure with full support. Finally, we gather in Section \ref{Section 6} some open questions and remarks.

\subsection{Duality} The following fact will be used several times in the paper. Each time we will write ``by duality'', this will refer implicitly to this fact.
\begin{fact}\label{dualityrem} Let (P) be a property of Banach space operators. If, for any $1<p<\infty$, a typical $T\in\mathcal B_1(\ell_p)$ satisfies (P), then it is also true that for any $1<p<\infty$, a typical $T\in\mathcal B_1(\ell_p)$ is such that $T^*$ satisfies (P).
\end{fact}
\bpf This is clear since the map $T\mapsto T^*$ is a homeomorphism from $\mathcal B_1(\ell_p)$ onto $\mathcal B_1(\ell_{q})$, where $q$ is the conjugate exponent of $p$. Of course, this argument works only because we are using the topology \sote\ (it would break down for the topologies \wot\ and \sot, for instance).
\epf
\section{Operators with maximal local spectra}\label{Section 3}
It is proved in \cite[Proposition 3.9]{GMM2} that a typical \(T\in(\mathcal{B}_{1}(\lp), \sot)\), \(1<p<\infty\) is such that \(\sigma (T)=\ba{\,\D}\); and the proof given there works for the topology \sote\ as well. 
Thus, typical contractions on \(\lp\) have ``maximal'' spectrum. The following result strengthens this statement and shows that typical contractions on \(\lp\) have maximal local spectra. This is to be compared with \cite[Theorem 1.5]{Vr}, where it is shown that for any Banach space $X$ and any $T\in\bx$ with SVEP, the set $\{ x\in X;\; \sigma_x(T)=\sigma(T)\}$ is residual in $X$.
\par\medskip
\begin{theorem}\label{Th1}
 Let $X=\ell_p$, $1<p<\infty$. A typical $T\in\bbx$ is such that $\sigma_x(T)=\overline{\,\D}$ for every $x\neq 0$. Equivalently, a typical $T\in\bbx$ has the following property: for any closed, $T\,$-$\,$invariant subspace $Z\subseteq X$, it holds that \(\sigma_x\bigl(T_{| Z}\bigr)=\overline{\,\D}\) for every $x\neq0$ in \(Z\). %In particular, 
\end{theorem}
 
The equivalence of the two statements is easy to check. Indeed, it follows directly from the definition of the local spectra that if $T\in\bx$ and if $Z$ is a closed $T\,$-$\,$invariant subspace of $X$, then $\sigma_x\bigl( T_{| Z}\bigr)\supseteq \sigma_x(T)$ for every $x\in Z$. Hence, if $\Vert T\Vert\leq 1$ and $\sigma_x(T)=\overline{\, \D}$ for some $x\in Z$, then \textit{a fortiori} 
$\sigma_x\bigl( T_{| Z}\bigr)=\overline{\,\D}$.

\medskip Before giving the proof of Theorem \ref{Th1}, we present some consequences.

\bco\label{noeigen}  Let $X=\ell_p$, $1<p<\infty$. A typical $T\in\bbx$ has the following properties : $T-\lambda$ is one-to-one with dense range for every $\lambda\in\C$, and $T-\lambda$ does not have closed range for any $\lambda\in\overline{\,\D}$. In particular, the essential spectrum of a typical $T\in\bbx$ is equal to $\overline{\,\D}$.
\eco
\bpf These results are known; see \cite[Proposition 6.3]{EM}, \cite[Proposition 2.24 and Remark 2.30]{GMM1} and \cite[Proposition 7.1]{GMM2}. However, they can also be deduced from Theorem \ref{Th1}. Indeed, since $\sigma_x(T)\subseteq\{ \lambda\}$ if $Tx=\lambda x$, it follows from Theorem \ref{Th1} that a typical $T\in\bbx$ has no eigenvalue. By duality (see Fact \ref{dualityrem}), this implies that a typical $T\in\bbx$ is such that $T-\lambda$ has dense range for every $\lambda\in\C$. Moreover, it also follows from Theorem \ref{Th1} that a typical $T\in\bbx$ is such that $\sigma(T)=\overline{\,\D}$; and so for a typical $T\in\bbx$, the operator $T-\lambda$ cannot have closed range for any $\lambda\in\overline{\,\D}$.
\epf

\smallskip
\begin{corollary}
 Let $X=\ell_p$, $1<p<\infty$. A typical $T\in\bbx$ is completely indecomposable.
\end{corollary}

\begin{proof}
  Since $\sigma_x\bigl(T_{| Z}\bigr)\subseteq \sigma\bigl(T_{| Z}\bigr)$ for any $T\,$-$\,$invariant subspace $Z\subseteq X$ and every $x\in Z$, it is clear from Theorem \ref{Th1} that a typical $T\in\bbx$ is strongly indecomposable. 
  By duality, it follows that a typical \(T\in \bbx\) is such that \(T^{*}\) is strongly indecomposable as well.
\end{proof}

As another consequence of Theorem \ref{Th1}, we are able to determine whether the local spectral properties presented in Section \ref{Section Background from local spectral theory} hold, or not, for a typical \(T\in\mathcal{B}_{1}(\lp)\):

\smallskip
\begin{corollary}
  Let $X=\ell_p$, $1<p<\infty$. A typical $T\in\bbx$ has the following properties.
\begin{enumerate}
 \item[\rm (i)] $r_x(T)=\sup\,\{ \vert \lambda\vert;\; \lambda\in \sigma_x(T)\}=1$ for every $x\neq 0$.
\item[\rm (ii)] For any $F\subseteq \C$, either $X_T(F)=\{ 0\}$ or $X_T(F)=X$; more precisely: $X_T(F)=X$ if $F$ contains $\overline{\,\D}$, and $X_T(F)=\{ 0\}$ otherwise.
\item[\rm (iii)] $T$ and $T^*$ have Property \emph{(C)}, and hence \emph{SVEP}.
\item[\rm (iv)] $T$ and $T^*$ have neither Property $(\beta)$, nor Property $(\delta)$.%, and hence are not decomposable.
\item[\rm (v)] \(T\) has empty localizable spectrum: \(\sigma _{loc}(T)=\emptyset\).
\end{enumerate}
\end{corollary}

\begin{proof}
 Note that if $\sigma_x(T)=\overline{\,\D}$ for all $x\neq 0$, then in particular $T$ has SVEP and hence $r_x(T)=\sup\,\bigl\{ \vert\lambda\vert;\; \lambda\in\sigma_x(T)\bigr\}$ for every $x\neq 0$. This proves (i).

\smallskip Part (ii) is clear.

\smallskip From (ii), we deduce that a typical $T\in\bbx$ has Property (C) and does not have Property $(\delta)$. Indeed, this is clear for (C). As for $(\delta)$, assume that $T$ satisfies (ii). Setting $U_1:=D(0,2/3)$ and $U_2:=\C\setminus \overline{D}(0,1/3)$, we get an open covering $(U_1,U_2)$ of $\C$ such that neither $\overline{\,U_1}$ nor $\overline{\,U_2}$ contains $\overline{\,\D}$; so we have $X_T(\overline{\,U_1})=\{ 0\}=X_T(\overline{\, U_2})$, which shows that $T$ does not have Property $(\delta)$.

Thus, for any $1<p<\infty$, a typical $T\in\mathcal B(\ell_p)$ has property (C) and does not have Property $(\delta)$. By duality, and since an operator $T$ has Property $(\beta)$ if and only if $T^*$ has Property $(\delta)$, this proves (iii) and (iv). Part {\rm (v)} follows immediately from {\rm (ii)}.
\end{proof}

\smallskip The next corollary shows that the \emph{algebraic core} and the \emph{analytic core} of a typical $T\in\bbx$ are far from being equal. Recall that if $T\in\bx$, the  {algebraic core of $T$}, denoted by $C(T)$, is the set of all $x\in X$ admitting a backward $T\,$-$\,$orbit, 
\mbox{\it i.e.} a sequence $(x_n)_{n\geq 0}$ such that $x_0=x$ and $Tx_n=x_{n-1}$ for all $n\geq 1$; and that the  {analytic core} of $T$, denoted by $K(T)$, is the set of all $x\in X$ admitting a ``controlled'' backward $T\,$-$\,$orbit, \mbox{\it i.e.} a backward $T\,$-$\,$orbit $(x_n)$ such that $\Vert x_n\Vert \leq \delta^n$ for some constant $\delta$ and all $n\geq 0$. Obviously, $K(T)\subseteq C(T)\subseteq \bigcap_{n\in\N} T^n(X)$.
\bco  Let $X=\ell_p$, $1<p<\infty$. A typical $T\in\bbx$ is such that $K(T)=\{ 0\}$ whereas $C(T)$ is dense in $X$.% and $C(T)=\bigcap_{n\geq 0} T^n(X)\neq X$. 
\eco
\bpf It is known that for any $T\in\bx$, we have $K(T)=\{x\in X;\; 0\not\in\sigma_x(T)\}$, see \mbox{e.g.} \cite[Proposition 1.3]{Mostafa}. So, by the previous corollary, a typical $T\in\bbx$ is such that $K(T)=\{ 0\}$. On the other hand, by Corollary \ref{noeigen}, a typical $T\in\bbx$ is one-to-one with dense range. Now, if $T$ is one-to-one then it is clear that $C(T)=\bigcap_{n\in\N} T^n(X)$; and if $T$ has dense range then, by the Bourbaki-Mittag-Leffler Theorem (see \mbox{e.g.} \cite{Runde}), 
$\bigcap_{n\in\N} T^n(X)$ is dense in $X$. This concludes the proof.%Finally, a typical $T\in\bbx$ is not onto
\epf

\smallskip Recall that if $T\in\bx$, then a \emph{spectral maximal subspace} for $T$ is a closed $T\,$-$\,$invariant subspace $Z\subseteq X$ with the following property: for any closed, $T\,$-$\,$invariant subspace $Z'\subseteq X$ such that $\sigma\bigl( T_{| Z'}\bigr)\subseteq \sigma\bigl( T_{| Z}\bigr)$, it holds that $Z'\subseteq Z$. (See \cite{CF} for   more on this notion.) Obviously, $\{ 0\}$ and $X$ are spectral maximal subspaces for $T$.
\begin{corollary}\label{Cor 3.4}
 Let $X=\ell_p$, $1<p<\infty$. A typical $T\in\bbx$ is such that $\sigma\bigl( T_{| Z}\bigr)=\overline{\,\D}$ for every closed $T\,$-$\,$invariant subspace $Z\neq\{ 0\}$. In particular,  the only spectral maximal subspaces of a typical $T\in\bbx$ are $\{ 0\}$ and $X$.
\end{corollary}

\begin{proof}
 The first part of the corollary follows immediately from Theorem \ref{Th1} since $\sigma \bigl( T_{| Z}\bigr)\supseteq \sigma_x\bigl( T_{| Z}\bigr)$ for any $T\,$-$\,$invariant subspace $Z$ and every $x\in Z$. For the second part, apply the definition with a maximal spectral subspace $Z\neq\{ 0\}$ and $Z':= X$.
\end{proof}

When \(X=\ell_{2}\), we have at our disposal a very strong result for proving the existence of invariant subspaces, namely the Brown-Chevreau-Pearcy Theorem \cite{BCP}. It states that any contraction on a Hilbert space whose spectrum contains the unit circle \(\T\) has a non-trivial invariant subspace. Combining it with Corollary \ref{Cor 3.4} above, we obtain:	

\begin{corollary}
 A typical $T\in\mathcal B_1(\ell_2)$ has no minimal invariant subspace except $\{ 0\}$: for any closed $T\,$-$\,$invariant  subspace $Z\neq\{ 0\}$, the operator $T_{| Z}$ has a non-trivial invariant subspace.
\end{corollary}
\par\medskip
We now move over to the proof of Theorem \ref{Th1}. For any Banach space \(X\), we introduce the set 
\[
\mathcal{G}:=\bigl\{T\in\bbx\;;\;\forall\,x\neq 0,\ \textrm{the set}\ \{\mu \in\D\;;\;x\not\in \textrm{Ran}(T-\mu )\}\ \textrm{is dense in}\ \D\bigr\}.
\]

We start with two general lemmas:

\begin{lemma}\label{Gisgd} 
 Assume that  $X$ is reflexive. Then, $\mathcal G$ is a $\gd$ subset of $(\bbx,\emph{\sote})$.
\end{lemma}

 This relies on the following fact.
\begin{fact}\label{C} For any $\mu\in\C$ and any closed ball $B\subseteq X$, the set \[ \mathcal C_{\mu, B}:=\{ (T,x)\in \bbx\times X;\; x\in (T-\mu)(B)\}\] is closed in $(\bbx,{\sote})\times (X,w)$.
\end{fact}

\begin{proof}[\it Proof of Fact \ref{C}] The set $\mathcal C_{\mu,B}$ is the projection along $B$ of 
\[ \mathbf C_{\mu,B}:=\bigl\{ (T,y,x)\in \bbx\times B\times X;\; (T-\mu)y=x\bigr\}.\]
Therefore, since the ball $B$ is weakly compact, it is enough to check that  $\mathbf C_{\mu,B}$ is closed in $(\bbx,\sote)\times (B,w)\times (X,w)$. Now, the map $(T,y)\mapsto (T-\mu) y$ is continuous from $(\bbx,\sote) \times (B,w)$ into $(X,w)$, because we have 
$\pss{x^*}{(T-\mu)y}=\pss{(T^*-\mu)x^*}{y}$ for any $x^*\in X^*$ and the map $(y^*,y)\mapsto \pss{y^*}{y}$ is continuous on $(K,\Vert\,\cdot\,\Vert)\times (B,w)$ for any bounded set $K\subseteq X^*$. So $\mathbf C_{\mu,B}$ is indeed closed in $(\bbx,\sote)\times (X,w)\times (B,w)$.
\end{proof}

\begin{proof}[\it Proof of Lemma \ref{Gisgd}]
For any $N\in\N$, let us denote by $B_N$ the closed ball $\overline B(0,N)\subseteq X$. Let also $(V_i)_{i\in\N}$ be a countable basis of (non-empty) open sets for $\D$. Then, the following equivalence holds true for any $T\in\bbx$:
\begin{align*}
T\not\in\mathcal G\iff \exists x\in X\setminus\{ 0\} \; \exists N\in\N\;&\exists i\in\N\;:\;\\
&\Bigl(\forall \mu\in V_i\;:\; (T,x)\in\mathcal C_{\mu, B_N} \Bigr).
\end{align*}%Let us now show that $\mathcal G$ is $\gd$ in $(\bbx,{\sote})$
Indeed, given $x\in X\setminus\{ 0\}$, each set $F_N:=\{\mu\in\D \,;\, x\in (T-\mu)B_N\}$ is easily seen to be closed in $\D$, so that $\Omega_N:=\D\setminus F_N=
\{\mu\in\D \,;\, x\not\in (T-\mu)B_N\}$ is open in $\D$. Since
$$\bigcap_{N\ge 1}\Omega_N= \{\mu\in\D \,;\, x\in \textrm{Ran}(T-\mu)\},$$ the Baire Category Theorem implies that 
$\bigcap_{N\ge 1}\Omega_N$ is dense in $\D$ if and only if $\Omega_N$ is dense in $\D$ for each $N\ge 1$. Hence
\begin{align*}
T\not\in\mathcal G\iff \exists x\in X\setminus\{ 0\} \; \exists N\in\N \; ;\; 
\{\mu\in\D \,;\, x\not\in (T-\mu)B_N\} \textrm{ is not dense in } \D;
\end{align*}
from which it follows that
\begin{align*}
T\not\in\mathcal G\iff \exists x\in X\setminus\{ 0\} \; \exists N\in\N\;&\exists i\in\N\;:\;\\
&\Bigl(\forall \mu\in V_i\;:\; (T,x)\in\mathcal C_{\mu, B_N} \Bigr).
\end{align*}
By Fact \ref{C}, the condition under brackets defines a closed subset of $(\bbx,\sote)\times (X,w)$. Since $X\setminus\{ 0\}$ is a $K_\sigma$ subset of $(X,w)$, it follows that $\bbx\setminus\mathcal G$ is $F_\sigma$ in $(\bbx,\sote)$.
\end{proof}

Our second lemma provides a large class of operators belonging to the set \(\mathcal{G}\).
\smallskip
\blm\label{Gbis} Let $T\in\bbx$. Then $T$ belongs to $G$ provided the following holds true: for every open set $V\neq\emptyset$ in $\D$, 
\[ \overline{\rm span}^{\, w^*}\left(\bigcup_{\mu\in V} \ker(T-\mu)^*\right)=X^*.\]
\elm
\begin{proof} This is essentially obvious from the definition of $\mathcal G$. Indeed, an operator $T\in\bbx$ does not belong to $\mathcal G$ if and only if 
one can find an open set $V\neq\emptyset$ in $\D$ such that 
$\bigcap_{\mu\in V} {\rm Ran}(T-\mu)\neq\{ 0\}$; and if this holds then $\overline{\rm span}^{\, w^*}\left(\bigcup_{\mu\in V} \ker(T-\mu)^*\right)\neq X^*$ since $\overline{\rm span}^{\,w^*}\left(\bigcup_{\mu\in V} \ker(T-\mu)^*\right)\subseteq \left(\bigcap_{\mu\in V} {\rm Ran}(T-\mu)\right)^\perp$. 
\end{proof}

\begin{remark}\label{holo} The assumption of Lemma \ref{Gbis} is satisfied as soon as the operator $T^*$ admits a $w^*$-$\,$spanning family of holomorphic eigenvector fields defined on the unit disk $\D$, \textit{i.e.} there exists a family $(E_{i})_{i\in I}$ of holomorphic eigenvector fields for $T^*$ defined on $\D$ such that %of holomorphic maps, $E_i:\D\to X^*$, such that $T^*E_i(\lambda)=\lambda E_i(\lambda)$ for each $i\in I$ and every $\lambda\in\D$, and 
$\overline{\rm span}^{\,w^*}\,\{ E_i(\lambda);\; i\in I, \;\lambda\in\D\}=X^*$. This follows from the Hahn-Banach Theorem and the identity principle for holomorphic functions.
\end{remark}

\par\medskip
We are now ready for the proof of Theorem \ref{Th1}.

\smallskip
\begin{proof}[Proof of Theorem \ref{Th1}] By the definition of $\mathcal G$, it is rather clear that if $T\in\mathcal G$ and $x\neq 0$, then $\sigma_x(T)$ contains $\D$, and hence $\sigma_x(T)=\overline{\,\D}$ since  $\sigma_x(T)$ is a closed subset of $\C$ contained in $\overline{\,\D}$. Indeed, let $\lambda\in\D$. Since $T\in\mathcal G$, we see that for every open neighborhood \(V\) of \(\lambda \) there exists \(\mu \in V\) such that \(x\not\in\textrm{Ran}(T-\mu)\). Hence there cannot exist \emph{any} function \(f:V\to X\) such that \((T-z)f(z)\equiv x\) on \(V\), and thus \(\lambda \not\in\rho _{x}(T)\). 
So, by Lemma \ref{Gisgd}, it is enough to show that the set $\mathcal G$ is \sote -$\,$dense in $\bbx$. 

\smallskip Let us denote by $(e_j)_{j\geq 0}$ the canonical basis of $X=\ell_p$; and for each $N\in\Z_+$, set $E_N:= [e_0,\dots ,e_N]$. To prove the \sote-denseness of $\mathcal G$, we show that any finite-dimensional operator $A\in\mathcal B(E_N)$ such that $\Vert A\Vert <1$ (viewed as an operator on $X=\ell_p$) can be approximated in the topology \sote\ by operators belonging to $\mathcal G$. 

\smallskip Let $M>N$ be a large integer, let $\delta$ be a small positive number, and let $T\in\mathcal B(X)$ be defined as follows:
\[ Te_j:=\left\{ \begin{array}{ll}
Ae_j + \delta e_{j+M+1}&\hbox{if $0\leq j\leq M$},\\
e_{j+M+1}& \hbox{if $j>M$}.
\end{array}\right.\] 

\smallskip\noindent Since $\Vert A\Vert<1$, we have $\Vert T\Vert<1$ if $\delta$ is small enough. Moreover, if $M$ is large enough and $\delta$ small enough, then $T$ is as close as we wish to $A$ in the topology \sote. So, by Remark \ref{holo}, it is now enough to show that the operator $T^*$ admits a spanning family of holomorphic eigenvector fields defined on $\D$. But this follows from the proof of \cite[Proposition 2.10]{GMM1}: indeed, with the notations of \cite[Proposition 2.10]{GMM1} and identifying $A$ with $P_MAP_M$, we have $T^*=B_{\pmb{\omega}, A}$, where the weight sequence $\pmb{\omega}=(\delta,\dots, \delta, 1, 1,1\dots)$ is such that $\underline{\lim}_{k\to\infty} \,\omega_{kM+l}\cdots \omega_{M+l}\omega_l=1>\Vert A\Vert$ for every $0\leq l\leq M$.
\end{proof}

\section{Asymptotic behavior of orbits}\label{Section 4}
The results in this section are concerned with the study of some properties of orbits of typical contractions on \(\lp\,\)-$\,$spaces. Recall that if \(T\in\bx\), the \emph{orbit} of a vector \(x\in X\) under the action of \(T\) is the set \(\{T^{n}x\;;\;n\in\Z_{+}\}\). The dynamics of typical operators \(T\in(\mathcal{B}_{1}(\ell_{2}),\sote)\) are already studied in some detail in \cite{GMM1}, where emphasis is put mostly on properties related to hypercyclicity. Here, we are interested in asymptotic properties of orbits of typical operators. Our results are strongly motivated by \cite[Chapter V]{Mu}, where several striking results valid for \emph{every} operator $T\in\bx$ are obtained. 

%\medskip
\subsection{Not too small orbits} 
We begin this section with a somewhat abstract statement which will be later on applied to several concrete situations.

\begin{proposition}\label{trueProp2}
   Let $X=\ell_p$, $1<p<\infty$.  Let also 
$\mathbf G$ be a $G_\delta$ subset of $(0,\infty)^\N$. Assume that $\mathbf G$ contains all sequences $\omega=(\omega_n)_{n\in\N}\in(0,\infty)^\N$ such that $\inf_{n\in\N} \omega_n>0$, and that $\mathbf G$ is \emph{upward closed} for the product ordering of $(0,\infty)^\N$, \mbox{\it i.e.} {\rm (}$\omega\in\mathbf G$ and $\omega\leq \omega'${\rm )} implies $(\omega'\in\mathbf G)$. Then, a typical $T\in\bbx$ is such that for every $x\neq 0$, the sequence $\bigl(\Vert T^nx\Vert\bigr)_{n\in\N}$ belongs to $\mathbf G$.
\end{proposition}

\smallskip The proof of Proposition \ref{trueProp2} relies on the next two lemmas.

\begin{lemma}
 \label{vrai?} Let $X=\ell_p$, $1<p<\infty$. The set of all $T\in\bbx$ such that $\forall x\neq 0\;:\; \inf_{n\in\N}\, \Vert T^nx\Vert >0$ is \emph{\sote}-$\,$dense in $\bbx$.
\end{lemma}

\begin{proof}
 Let us denote by $\mathcal D$ the considered set of operators. As in the proof of Theorem \ref{Th1}, given a finite-dimensional $A\in\mathcal B(E_N)$, we consider 
 the operator $T$ defined by 
\[ Te_j:=\left\{ \begin{array}{ll}
Ae_j + \delta e_{j+M+1}&\hbox{if $0\leq j\leq M$},\\
e_{j+M+1}& \hbox{if $j>M$},
\end{array}\right.\] 

\noindent where $M>N$ and $\delta>0$. This operator $T$ belongs to $\bbx$ and is \sote -$\,$close to $A$ if $M$ is large enough and if $\delta$ is small enough. So it is enough to show that $T$ belongs to \(\mathcal{D}\), \mbox{\it i.e.} that $\inf_{n\in\N}\, \Vert T^nx\Vert >0$ for every $x\neq 0$.

\smallskip Let $x\in X\setminus\{ 0\}$. If $\pss{e_{j_0}^*}{x}\neq 0$ for some $j_0>M$, then, since $\pss{e_{j_0+n(M+1)}^*}{T^nx}=\pss{e_{j_0}^*}{x}$ for every $n\in\N$ by the definition of $T$, we see that $\inf_{n\in\N}\, \Vert T^nx\Vert >0$. If $\pss{e_{j}^*}{x}= 0$ for all $j>M$, \mbox{\it i.e.} if $x\in E_M$, then $x':=Tx$ satisfies $\pss{e_{j_0}^*}{x'}\neq 0$ for some $M<j_0\leq 2M+1$ by the definition of $T$; so we have $\inf_{n\in\N}\, \Vert T^nx'\Vert >0$, \mbox{\it i.e.} $\inf_{n\geq 2}\, \Vert T^nx\Vert >0$, and hence $\inf_{n\in\N}\, \Vert T^nx\Vert >0$ since $Tx=x'\neq 0$. This concludes the proof of the lemma.
\end{proof}

\begin{lemma}\label{faux?}
  If $\mathbf H\subseteq (0,1]^\N$ is $\gd$ and upward closed, then \(\textbf{H}\) can be written as $\mathbf H=\bigcap_{k\in\N} \mathbf O_k$, where the sets $\mathbf O_k$ are open in $(0,1]^\N$ and also upward closed.
\end{lemma}

\begin{proof}
  Write $\mathbf H=\bigcap_{k\in\N} \mathbf U_k$ for some open sets $\mathbf U_k\subseteq (0,1]^\N$. For each $k\in\N$, set 
\[ \mathbf O_k:=\bigl\{ \omega\in (0,1]^\N;\; \forall \omega'\geq \omega\;:\; \omega'\in\mathbf U_k\bigr\}.\]
The sets $\mathbf O_k$ are obviously upward closed, and since $\mathbf H$ is upward closed, we have $\mathbf H=\bigcap_{k\in\N} \mathbf O_k$. So we just need to check that each $\mathbf O_k$ is open in $(0,1]^\N$. Now, since $[0,1]^\N$ is compact and $(0,1]^\N$ is upward closed in $[0,1]^\N$, it is easily checked that $(0,1]^\N\setminus \mathbf O_k$ is closed in $(0,1]^\N$.
\end{proof}

We are now ready for the proof of Proposition \ref{trueProp2}.

\begin{proof}[\it Proof of Proposition \ref{trueProp2}]  By homogeneity and upward closedness of $\mathbf G$, it is enough to show that a typical $T\in\bbx$ is such that the sequence $(\Vert T^nx\Vert)_{n\in\N}$ belongs to $\mathbf G$ for every $x\neq 0$ with $\Vert x\Vert\leq 1$. Moreover, if $T\in\bbx$ and $\Vert x\Vert \leq 1$, then $\Vert T^n x\Vert\leq 1$ for all $n\in\N$. So, setting $\mathbf H:=\mathbf G\cap (0,1]^\N$ and denoting by $B_X$ the closed unit ball of $X$, we have to show that the set 
\[ \mathcal H:=\Bigl\{ T\in\bbx;\; \forall x\in B_X\setminus\{ 0\}\;:\; \bigl(\Vert T^nx\Vert\bigr )_{n\in\N}\in\mathbf  H\Bigr\}\]
is comeager in $\bbx$. 

By Lemma \ref{vrai?} and by assumption on $\mathbf G$, we already know that the set $\mathcal H$ is dense in $\bbx$. We show that $\mathcal H$ is also $G_\delta$.

\smallskip By Lemma \ref{faux?}, we may write \(\textbf{H}\) as $\mathbf H=\bigcap_{k\in\N} \mathbf O_k$, where the sets $\mathbf O_k$ are open and upward closed in $(0,1]^\N$. So it is enough to check that if $\mathbf O\subseteq(0,1]^\N$ is open and upward closed, then the set 
\[ \mathcal O:=\Bigl\{ T\in\bbx;\; \forall x\in B_X\setminus\{ 0\}\;:\; \bigl(\Vert T^nx\Vert\bigr )_{n\in\N}\in\mathbf O\Bigr\}\]
is $\gd$ in $\bbx$.

Let us denote by $\Sigma$ the set of all finite sequences from $(0,1]$; and for any sequence $\sigma=(s_1,\dots ,s_r)\in\Sigma$, let us set 
\[ \mathbf V_\sigma:=\{ \omega\in (0,1]^\N;\; \omega_n>s_n\;\hbox{for $n=1,\dots ,r$}\}.\]

\noindent Since $\mathbf O$ is upward closed and open in $(0,1]^\N$, one can write
\[ \mathbf O=\bigcup_{\sigma\in I} \mathbf V_\sigma,\]
for some set $I\subseteq \Sigma$. Indeed, $\mathbf O$ is a union of basic open sets of the form $\bigcap_{n=1}^r \mathbf W_{n, s_n, t_n}$, where we set $\mathbf W_{n,s,t}:=\{\omega;\; s<\omega_n<t\}$ for any $n\in\N$ and any real numbers $s,t$ such that $0<s\leq 1$ and $s<t$.  Now, if 
$t_1,\ldots, t_r$ are given such that $t_n>s_n$ for every $n=1,\ldots, r$,
then, for every $\omega\in \mathbf V_\sigma$, one can find $\omega'\in \bigcap_{n=1}^r \mathbf W_{n, s_n, t_n}$ such that $\omega\geq\omega'$. Hence, by upward closedness, we have that if $ \bigcap_{n=1}^r \mathbf W_{n, s_n, t_n}\subseteq\mathbf O$, then in fact $\mathbf V_\sigma\subseteq\mathbf O$; which proves the claim.

 So, for any $T\in\bbx$, we see that
\begin{align*}T\not\in \mathcal O\iff \exists x\in B_X&\setminus \{ 0\} \\
& \forall \sigma=(s_1,\dots ,s_n)\in I\;:\; \Bigl(\exists n\in \llbracket 1,r\rrbracket\;:\; \Vert T^nx\Vert \leq s_n\Bigr)
\end{align*}

Now, observe that for any bounded set $B\subseteq X$, the map $(A,x)\mapsto \Vert Ax\Vert$  is lower semi-continuous on $(\bbx,\sote)\times (B, w)$, because the map $(A,x)\mapsto Ax$ is continuous from $(\bbx,\sote)\times (B, w)$ into $(X,w)$ and the norm is lower semi-continuous on $(X,w)$. Since $B_X\setminus\{ 0\}$ is $K_\sigma$ in $(X,w)$, it follows easily that $\bbx\setminus\mathcal O$ is $F_\sigma$ in $\bbx$.
\end{proof}

A first consequence of Proposition \ref{trueProp2} is that all non-zero orbits of a typical \(T\in\mathcal{B}_{1}(\lp)\) are ``not too small''.

\begin{corollary}\label{Prop2} Let $X=\ell_p$, $1<p<\infty$. 
\begin{enumerate}
\item[\rm (i)] Given a sequence of positive real numbers $(a_n)$ such that $a_n\to 0$, a typical $T\in\bbx$ %$T\in (\bbx,\emph{\sote})$ 
is such that, for every $x\neq 0$, one has $\Vert T^n x\Vert >a_n$ for infinitely many $n\in\N$.
\item[\rm (ii)] Given a sequence $(\phi_n)$ of increasing continuous functions from $\R^+$ into itself such that $\sum_{n=0}^\infty \phi_n(\varepsilon)=\infty$ for every $\varepsilon >0$, a typical $T\in\bbx$  
is such that $$\sum_{n=0}^\infty \phi_n\bigl(\Vert T^n x\Vert\bigr)=\infty\quad\textrm{for every }x\neq 0.$$
\end{enumerate}
\end{corollary}

\bpf For (i), take $\mathbf G_1:=\bigl\{ \omega\in(0,\infty)^\N;\; \omega_n>a_n\;\hbox{for infinitely many $n$}\bigr\}$, which is obviously $\gd$ and upward closed. For (ii), consider $\mathbf G_2:=\bigl\{ \omega\in(0,\infty)^\N;\; \sum_{n=0}^\infty \phi(\omega_n)=\infty\bigr\}$, which is $\gd$ because $\omega\in\mathbf G_2\iff\forall M\in\N\;\exists N\;:\; \sum_{n=0}^N \phi_n(\omega_n)> M$, and upward closed because the functions $\phi_n$ are increasing.
\end{proof}

\begin{remark}\label{rxt}  Taking \mbox{e.g.} $a_n:=\frac1n$ in (i), one gets another proof that a typical operator $T\in\bbx$ is such that $r_x(T)=1$ for every $x\neq 0$.
\end{remark}

\subsection{Nilpotent operators} We now turn to a result that, sadly enough, we can prove only for $X=\ell_2$. Let us denote by $\mathcal N(X)$ the set of all nilpotent operators $T\in\bx$, and set $\mathcal N_1(X):=\mathcal N(X)\cap\bbx$.

\begin{proposition}\label{nilpot} 
If $X=\ell_2$, then $\mathcal N_1(X)$ is \emph{\sote}-$\,$dense in $\bbx$.
\end{proposition}

\smallskip 
Before giving the proof of Proposition \ref{nilpot}, we explore some of its consequences. First, it implies that a typical $T\in\mathcal B(\ell_2)$ has lots of ``partly small'' orbits (see Proposition \ref{powerreg} and Proposition \ref{Prop 4.14} below for concrete examples):

\begin{corollary}\label{Ginverse} Let $X=\ell_2$. Let also $\mathbf G$ be a $\gd$ subset of $[0,\infty)^\N$. Assume that $\mathbf G$ contains all sequences $\omega=(\omega_n)\in [0,\infty)^\N$ which are eventually $0$. Then, a typical $T\in\bbx$ has the following property: for a comeager set of vectors $x\in X$, the sequence $\bigl(\Vert T^nx\Vert\bigr)_{n\in\N}$ belongs to $\mathbf G$. 
\end{corollary}

\begin{proof} Let $D$ be a countable dense subset of $X$, and define
\[ \mathcal G:=\bigl\{ T\in\bbx;\; \forall z\in D\;:\; (\Vert T^nz\Vert)_{n\in\N}\in\mathbf G\bigr\}.\]

 Since $\mathbf G$ is $\gd$, it is easily checked that $\mathcal G$ is \sot$\,$-$\,\gd$, hence \sote -$\,\gd$. Moreover, $\mathcal G$ contains all nilpotent operators in $\bbx$ by assumption on $\mathbf G$, and hence $\mathcal G$ is \sote -$\,$dense in $\bbx$ by Proposition \ref{nilpot}. Thus, we see that a typical $T\in\bbx$ is such that the set $D_T:=\bigl\{ x\in X;\; (\Vert T^nx\Vert)_{n\in\N}\in\mathbf G\bigr\}$ contains $D$, and hence is dense in $X$. But for any fixed $T\in\bbx$, the set $D_T$ is clearly $\gd$; so the proof is complete.
\end{proof}

\begin{remark} By the Kuratowski-Ulam Theorem, it follows from Corollary \ref{Ginverse} that the set $\bigl\{ (T,x)\in\bbx\times X;\;( \Vert T^nx\Vert)_{n\in N}\in\mathbf G\bigr\}$  is comeager in $\bigl(\bbx,\sote\bigr)\times X$, and also in $\bigl(\bbx, \sot\bigr)\times X$. However this set has no reason for being $\gd$.
\end{remark}

\smallskip Let us point out another consequence of Proposition \ref{nilpot}.

\begin{corollary}\label{smallspectrum} Let $X=\ell_2$. For any $a\in\D$, the set of all $T\in\bbx$ such that $\sigma(T)=\{ a\}$ is \emph{\sote}-$\,$dense in $\bbx$.
\end{corollary}

\begin{proof} Let us denote by $\varphi_a$ the M\"obius transformation associated with $a$,
\[ \varphi_a(z)=\frac{a-z}{1-\bar a z}\cdot\]
Since $\varphi_a$ is holomorphic in a neighborhood of $\overline{\,\D}$, one can define $\varphi_a(T)$ for any $T\in\bbx$. Moreover, the map $T\mapsto\varphi_a(T)$ from $\bbx$ into itself is \sote -$\,$continuous because the power series of $\varphi_a$ is absolutely convergent on $\overline{\,\D}$; and since $\varphi_a\circ\varphi_a(z)\equiv z$ in a neighborhood of $\overline{\,\D}$, we have $\varphi_a\bigl(\varphi_a(T)\bigr)=T$ for every $T\in\bbx$. Finally, we also have that $\varphi_a(T)\in\bbx$ for any $T\in\bbx$, by von Neumann's inequality. So we may conclude that the map $T\mapsto \varphi_a(T)$ is an involutive homeomorphism of $\bbx$ onto itself. Since $\sigma(T)=\{a\}$ if and only if $\sigma\bigl(\varphi_a(T)\bigr)=\{ 0\}$ and since the set $\bigl\{ S\in\bbx;\; \sigma(S)=\{ 0\}\bigr\}$ is \sote -$\,$dense in $\bbx$ by Proposition \ref{nilpot}, the result follows.
\end{proof}

\medskip The proof of Proposition \ref{nilpot} relies on the next two lemmas.

\smallskip
\begin{lemma}\label{GS} Let $T$ be a bounded operator on a Banach space \(X\). If $E: \Omega\to X$ is a holomorphic eigenvector field for $T$ defined on some connected open set $\Omega\subseteq\C$ containing $0$, 
then $$\overline{\rm span}\,\bigl\{ E(\lambda);\; \lambda\in \Omega\bigr\}\subseteq\overline{\bigcup_{k\in\N} \ker(T^k)}.$$
\end{lemma}

\begin{proof} Note first that since $\Omega$ is connected, we have \[ \overline{\rm span}\,\bigl\{ E(\lambda);\; \lambda\in \Omega\bigr\}=\overline{\rm span}\,\bigl\{ E^{(n)}(0);\; n\geq 0\bigr\}.\]

\noindent Indeed, the inclusion $\supseteq$ is clear; and the other inclusion follows from the Hahn-Banach Theorem and the identity principle for holomorphic functions. Next, differentiating the identity $TE(\lambda)\equiv \lambda E(\lambda)$, we get that $TE^{(n)}(\lambda)=\lambda E^{(n)}(\lambda)+n E^{(n-1)}(\lambda)$ for every \(\lambda \in\Omega \) and for all $n\geq 1$. In particular, $TE^{(n)}(0)= n E^{(n-1)}(0)$ for all $n\geq 1$. Since $T E(0)=0$, it follows that $T^{n+1}E^{(n)}(0)=0$ for all $n\geq 0$. So we see that ${\rm span}\,\bigl\{ E^{(n)}(0);\; n\geq 0\bigr\}\subseteq \bigcup_{k\in\N} \ker(T^k)$, which concludes the proof of the lemma.
\end{proof}

As an immediate corollary of Lemma \ref{GS}, we obtain

\begin{corollary}\label{Autre corollaire} If an operator $T\in\bx$ admits a spanning family of holomorphic eigenvector fields $(E_i)_{i\in I}$ defined on some connected open set $\Omega\subseteq\C$ containing $0$, \mbox{\it i.e.} $\overline{\rm span}\,\{ E_i(\lambda);\; i\in I, \lambda\in\Omega\}=X$, then $\bigcup_{k\in\N} \ker(T^k)$ is dense in $X$.
\end{corollary}

\smallskip
Our second lemma is specific to the Hilbertian setting.

\begin{lemma}\label{nilpadh} Let $X=\ell_2$. If $T\in\bbx$ is such that $\bigcup_{k\in\N} \ker(T^k)$ is dense in $X$, then $T$ belongs to the \emph{\sote}-\,closure of \(\mathcal{N}_{1}(X)\).
\end{lemma}

\begin{proof} For each $k\in\N$, let us denote by $Q_k\in\bx$ the orthogonal projection of \(X\) onto $\ker(T^k)$. Then $\Vert Q_k\Vert\leq 1$ (we use specifically here the fact that \(X=\ell_{2}\)), and since the sequence of subspaces $\bigl( \ker(T^k)\bigr)_{k\in\N}$ is increasing with $\overline{\bigcup_{k\in\N}\ker(T^k)}=X$, we see that $Q_k\xrightarrow{\sot} I$. Moreover, the projections $Q_k$ are self-adjoint, so in fact $Q_k\xrightarrow{\sote} I$. Hence, if we set $T_k:= TQ_k$, then $T_k\in\bbx$ for all $k\in\N$ and $T_k\xrightarrow{\sote}T$. Moreover, each operator $T_k$ is nilpotent, in fact $T_k^{k}=0$: indeed, since $T({\rm Ran}(Q_k))=T(\ker(T^k))\subseteq \ker(T^k)={\rm Ran}(Q_k)$, we have $Q_kTQ_k=TQ_k$, so $T_k^n=T^nQ_k$ for all $n\in\N$ and hence $T_k^k=0$ since ${\rm Ran}(Q_k)=\ker(T^k)$. So we have found a sequence $(T_k)\subseteq\mathcal N_1(X)$ such that $T_k\xrightarrow{\sote}T$. 
\end{proof}

Proposition \ref{nilpot} can now be readily deduced from Lemmas \ref{GS} and \ref{nilpadh}.
\smallskip

\begin{proof}[\it Proof of Proposition \ref{nilpot}] By Lemmas \ref{GS} and \ref{nilpadh}, it is enough to show that the set of all $T\in\bbx$ admitting a spanning family of holomorphic eigenvector fields defined on $\D$ is $\sote$-$\,$dense in $\bbx$. But this follows from the proof of \cite[Corollary 2.12]{GMM1}. Indeed, this proof shows that for any $M>1$, the set of all $T\in\mathcal B_M(X)$ admitting a spanning family of holomorphic eigenvector fields defined on the open disk $D(0,M)$ is \sote -$\,$dense in $\mathcal B_M(X)$; so we get the required result  by homogeneity.
\epf

\smallskip
\begin{remark} Here is a completely different and highly non-elementary proof of Pro\-position \ref{nilpot}. We prove in fact the following stronger result: \emph{A typical $T\in\bbx$ belongs to the {norm-closure} of $\mathcal N_1(X)$}.

\smallskip By a deep result of Apostol-Foias-Voiculescu \cite{AFV}, an operator $T\in\bx$ belongs to the norm-closure of nilpotent operators if and only if it has the following properties: 

\begin{itemize}
\item[]
\be
\item[\rm (i)] $\sigma(T)$ and $\sigma_e(T)$ are connected with $0\in\sigma_e(T)$.
\item[\rm (ii)]  ${\rm ind}(T-\lambda)=0$ for any 
$\lambda\in\C$ such that $T-\lambda$ is a semi-Fredholm operator.
\ee
\end{itemize}

\noindent
Now, by Corollary \ref{noeigen}, a typical $T\in\bbx$ certainly satisfies (i), and a typical $T\in\bbx$ also satisfies (ii) vacuously for $\lambda\in\overline{\,\D}$. Since $T-\lambda$ is invertible if $\vert \lambda\vert>1$, the result follows.
\end{remark}

%\medskip 

%\smallskip 
\subsection{Power-regular operators} According to \cite{Atz}, an operator $T\in\bx$ is said to be \emph{power-regular} if $\Vert T^nx\Vert^{1/n}\to r_x(T)$ as $n\to\infty$ for every $x\in X$. It is shown in \cite{Atz} that any decomposable operator is power-regular. Combining Theorem \ref{Th1} and Corollary \ref{Ginverse}, we obtain the following result.

\begin{proposition}\label{powerreg} Let $X=\ell_2$. A typical $T\in\bbx$ is not power-regular. More precisely, a typical $T\in\bbx$ is such that $r_x(T)=1$ for all $x\neq 0$ yet $\underline{\lim}\, \Vert T^nx\Vert^{1/n}= 0$ for a comeager set of vectors $x\in X$.
\end{proposition}

\begin{proof} By Theorem \ref{Th1} (or Remark \ref{rxt}), a typical $T\in\bbx$ is such that $r_x(T)=1$ for every $x\neq 0$. On the other hand,  a typical $T\in\bbx$ is such that $\underline{\lim}\, \Vert T^nx\Vert^{1/n}=0$ for a comeager set of vectors $x\in X$: this follows from Corollary \ref{Ginverse} applied to the set \[ \mathbf G:=\bigl\{ \omega=(\omega_n)\in[0,\infty)^\N;\; \underline{\lim}\, \omega_n^{1/n}=0\bigr\}.\]
\end{proof}

\subsection{Distributionally chaotic operators} 
Let us recall that a vector $x\in X$ is said to be \emph{distributionally irregular} for an operator $T\in\bx$ if there exist two sets $A,B\subseteq \N$ both having upper density equal to $1$, such that $\Vert T^nx\Vert\to 0$ as $n\to \infty$ along $A$ and $\Vert T^nx\Vert\to\infty$ as $n\to\infty$ along $B$. The operator $T$ is said to be \emph{densely distributionally chaotic} if it has a dense set of distributionally irregular vectors. We refer to \cite{BBMP} for more on these notions.

\medskip Note that for $T$ to have any distributionally irregular vector, it is necessary that $\Vert T\Vert >1$. It is shown in \cite{GMM1} that for any $M>1$, a typical $T\in\mathcal B_M(\ell_2)$ is densely distributionally chaotic. This result can be slightly improved, as follows:

\begin{proposition}\label{Prop 4.14} Let $X=\ell_2$. For any $M>1$, a typical $T\in\bmx$ has the following properties: for every $x\neq 0$, there is a set $A\subseteq\N$ with $\overline{\rm dens}(A)=1$ such that $\Vert T^nx \Vert\to \infty$ as $n\to\infty$ along $A$; and for a comeager set of vectors $x\in X$, there is a set $B\subseteq\N$ with $\overline{\rm dens}(B)=1$ such that $\Vert T^n x\Vert\to 0$ as $n\to\infty$ along $B$.
\end{proposition}

\begin{proof} Let us first show that a typical $T\in\bmx$ is such that for every $x\neq 0$, there exists a set $A\subseteq\N$ with $\overline{\rm dens}(A)=1$ such that $\Vert T^nx\Vert\to \infty$ as $n\to\infty$ along $A$. Choose $\alpha\in (0,1)$ such that $M\alpha >1$, and set
\[ \mathbf G:=\Bigl\{ \omega=(\omega_n)\in (0,\infty)^\N;\; \overline{\rm dens}\,\bigl(\{ n\in\N;\; \omega_n>\alpha^n\}\bigr)=1\Bigr\}.\]

\noindent The set $\mathbf G$ is easily seen to be $\gd$ in $(0,\infty)^\N$. Moreover, $\mathbf G$ is clearly upward closed, and since $\alpha<1$ it contains all sequences $\omega$ such that $\inf_n\omega_n>0$. By Proposition \ref{trueProp2}, it follows that a typical $T\in\bbx$ is such that for every $x\neq 0$, the set $\{ n\in\N;\; \Vert T^nx\Vert >\alpha^n\}$ has upper density equal to $1$. Hence, by homogeneity, a typical $T\in\bmx$ is such that for every $x\neq 0$, the set $A_{x,T}:=\{ n\in\N;\; \Vert T^nx\Vert >(M\alpha)^n\}$ has upper density equal to $1$; and clearly $\Vert T^nx\Vert\to\infty$ as $n\to\infty$ along $A_{x,T}$.

\smallskip Now, let us show that a typical $T\in\bmx$ is such that for a comeager set of vectors $x\in X$, there exists a set $B\subseteq\N$ with $\overline{\rm dens}(B)=1$ such that $\Vert T^nx\Vert\to 0$ as $n\to\infty$ along $B$. This follows in fact from \cite[Proposition 2.40]{GMM1}, but the proof we give here is rather different and more elementary. Choose $\beta >0$ such that $\beta M<1$. Applying Corollary \ref{Ginverse} to the set 
\[ \mathbf G:=\Bigl\{ \omega=(\omega_n)\in[0,\infty)^\N;\;  \overline{\rm dens}\,\bigl(\{ n\in\N;\; \omega_n <\beta^n\}\bigr)=1\Bigr\},\]
we see that a typical $T\in\bbx$ is such that $\overline{\rm dens}\,\bigl(\{ n\in\N;\; \Vert T^nx\Vert <\beta^n\}\bigr)=1$ for a comeager set of vectors $x\in X$. By homogeneity, it follows that a typical $T\in\bmx$ has the following property: for a comeager set of vectors $x\in X$, the set \[B_{T,x}:=\{ n\in\N;\; \Vert T^nx\Vert <(\beta M)^n\}\] has upper density equal to $1$. Since obviously $\Vert T^nx\Vert\to 0$ as $n\to\infty$ along $B_{T,x}$, this concludes the proof of Proposition \ref{Prop 4.14}.
\end{proof}

\begin{remark} A slight modification of the above proof gives the following result. Let $(\alpha_n)$ be any sequence of positive real numbers tending to $0$, and let $(\beta_n)$ be any sequence of positive real numbers. Then, a typical $T\in\bmx$ has the following properties: for every $x\neq 0$, there is a set $A\subseteq\N$ with $\overline{\rm dens}(A)=1$ such that $\Vert T^nx \Vert >\alpha_n M^n$ for all $n\in A$; and for a comeager set of vectors $x\in X$, there is a set $B\subseteq\N$ with $\overline{\rm dens}(B)=1$ such that $\Vert T^n x\Vert < \beta_n$ for all $n\in B$.
\end{remark}

\begin{remark} It is \emph{not} true that given \(M>1\), a typical $T\in\bmx$ is such that every $x\neq 0$ is a distributionally irregular vector for $T$. Indeed, by \cite[Corollary V.37.9]{Mu}, for any $T\in\bx$, there is a dense set of vectors $x\in X$ such that $\Vert T^n x\Vert^{1/n}\to r(T)$ as $n\to\infty$. Since a typical $T\in \bmx$ is such that its spectral radius is equal to \(M\), it follows that for a typical $T\in \bmx$, there is a dense set of vectors $x\in X$ such that $\Vert T^nx\Vert\to\infty$, and hence a dense set of vectors $x$ which are not distributionally irregular for $T$.
\end{remark}

\section{Generalized kernels and Gaussian mixing}\label{Section 5} As we have seen, sets of the form \(\ker^*(T):=\bigcup_{k\in\N}\ker (T^{k})\), where $T$ is a bounded operator on a complex separable Banach space $X$, play a prominent role in the proof of Proposition \ref{nilpot}. The set $\ker^*(T)$ is usually called the \emph{generalized kernel} of the operator $T$. Operators with a $1\,$-$\,$dimensional kernel and a dense generalized kernel are called \emph{generalized backward shifts} in \cite{GS}, where it is shown that operators commuting with generalized backward shifts have remarkable dynamical properties. We would like to point out the following kind of converse to Corollary \ref{Autre corollaire}, which is very much in the spirit of \cite[Theorem 3.6]{GS}. The proof is essentially the same as that of \cite[Theorem 2]{F}, but we give the details for convenience of the reader.

\begin{proposition}\label{GS2} Let $X$ be a complex separable Banach space, and let $T\in\bx$. Assume that $T$ is \emph{onto} and that there exists an operator $B$ with dense generalized kernel such that
$TB=BT$ and $\ker(B)\subseteq\ker(T)$.  
Then, $T$ admits a spanning family of holomorphic eigenvector fields defined on the disk $D(0,c_T)$, where $c_T:=\inf\,\{ \Vert T^*x^*\Vert;\; \Vert x^*\Vert =1\}$. 
\end{proposition}
\bpf  The quantity $c_T$ is positive because $T$ is onto; and we know that for any $x\in X$ and any $c>c_T^{-1}$, one can find $x'\in X$ such that $\Vert x'\Vert\leq c\,\Vert x\Vert$ and $Tx'=x$.

\smallskip Note also that the assumptions imply that $\ker(B^k)\subseteq\ker(T^k)$ for every $k\in\N$. So $T$ itself has a dense generalized kernel.% (and hence we could completely forget $B$

\smallskip
Let $Z:=\ker^*(T)$. For any $z\in Z$, let us denote by $k_z$ the smallest integer $k\geq 1$ such that $T^kz=0$.  By the definition of $c_T$, one can find a sequence $(z_j)_{j\geq k_z-1}$ of vectors of $Z$ with $z_{k_z-1}=z$ such that $Tz_j=z_{j-1}$ and $\Vert z_j\Vert \leq (1+2^{-j}) c_T^{-1}\,\Vert z_{j-1}\Vert$ for all $j\geq k_z$. Set also $z_j:= T^{k_z-1-j} z$ for $0\leq j <k_z-1$, so that $Tz_j=z_{j-1}$ for every $j>0$. Then, for any $\lambda\in D(0,c_T)$, the series $\sum_{j\geq 0} \lambda^{j}  z_{j}$ is convergent; so the formula 
\[ E_z(\lambda):= \sum_{j\geq 0} \lambda^{j} z_{j}  \]
defines a holomorphic function $E_z:D(0,c_T)\to X$. Note that $Tz_0=0$ since $z_0=T^{k_z-1} z$. %and $z\in\ker(B^{k_z})\subseteq\ker(T^{k_z})$. 
By the choice of the sequence $(z_j)$,  it  follows that $TE_z(\lambda)=\lambda E_z(\lambda)$ for every $\lambda\in D(0,c_T)$. So we have defined on $D(0,c_T)$ a family of holomorphic eigenvector fields $(E_z)_{z\in Z}$ for $T$. Moreover, this family is spanning. Indeed, if $x^*\in X^*$ is such that $\pss{x^*}{E_z(\lambda)}=0$ for every $z\in Z$ and $\lambda\in D(0, c_T)$, then $\pss{x^*}{z}=\pss{x^*}{z_{k_z-1}}=0$ for every $z\in Z$, and hence $x^*=0$ since $Z$ is dense in $X$. 
\epf

\begin{remark} Combining Lemma \ref{GS} with the proof of Proposition \ref{GS2}, we obtain the following result: if $T\in\bx$ is onto, then there exists a family of holomorphic eigenvector fields $(E_i)_{i\in I}$ defined on $D(0,c_T)$ such that $\overline{\rm span}\,\{ E_i(\lambda);\, i\in I, \lambda\in D(0,c_T)\}=\overline{\ker^*(T)}$.
\end{remark}

\smallskip
We now use Proposition \ref{GS2} to provide a condition under which an operator $T\in\bx$ is \emph{mixing in the Gaussian sense}, which means that $T$ admits an invariant Gaussian probability  measure with full support with respect to which it is a strongly mixing transformation. (We refer to \cite[Chapter 5]{BMbook} and \cite{BG} for unexplained terminology and for more about the ergodic theory of linear dynamical systems.) This result is essentially proved in \cite[Corollary 9]{Mu2}, but our formulation is slightly more general.

\bco\label{Mullererie} Let $X$ be a complex separable Banach space, and let $T\in\bx$. Assume that $T$ has the form $T=f(B)$, where $B\in\bx$ is onto with a dense generalized kernel, and $f$ is a function holomorphic on an open set $\Omega\supseteq \sigma(B)$, not constant on any connected component of $\Omega$, and such that 
$\vert f(z_0)\vert=1$  for some $z_0\in \Omega\cap D(0,c_B)$. Then, the operator 
$T$ is mixing in the Gaussian sense. 
\eco
\bpf Let $V\subseteq \Omega\cap D(0,c_B)$ be a connected open neighborhood of $z_0$. By assumption, the function $f$ is not constant on $V$, so $f(V)\cap\T$ is a non-empty open subset of $\T$. In particular $f(V)\cap\T$ is uncountable, so $V\cap f^{-1}(\T)$ is uncountable. On the other hand, the set $\{ z\in V;\; f'(z)=0\}$ is countable. So $W:=V\cap\{ z\in V;\; f'(z)\neq 0\}$ is a non-empty open set such that $f(W)\cap\T\neq\emptyset$. By the implicit function Theorem, it follows that one can find a non-trivial closed arc $\Lambda\subseteq\T$ and a (one-to-one) Lipschitz map $\phi:\Lambda\to V$ such that $f(\phi(\lambda))=\lambda$ for every $\lambda\in\Lambda$. 

By Proposition \ref{GS2}, the operator $B$ admits a spanning family of holomorphic eigenvector fields $(E_i)_{i\in I}$ defined on $D(0,c_B)$. If we set $\widetilde{E_i}(\lambda):=E_i(\phi(\lambda))$, we obtain a family of Lipschitz eigenvector fields for $T=f(B)$ defined on the arc 
$\Lambda$. Moreover, the family $(\widetilde{E_i})_{i\in I}$ is spanning by the Hahn-Banach Theorem and the identity principle for holomorphic functions, because $\widetilde{E_i}(\Lambda)= E_i(\phi(\Lambda))$ for all $i\in I$ and $\phi(\Lambda)$ is an infinite compact subset of $D(0,c_B)$. By \cite[Theorem 3.4]{BG}, it follows that $T$ is mixing in the Gaussian sense.\epf

\begin{example}\label{petit}  Let $X$ be a complex separable Banach space, and let  $B\in\bx$ be onto with a dense generalized kernel. For any $\lambda_0\in\C$ such that $\dist(\lambda_0, \T)<c_B$  (in particular, for any $\lambda_0\in\T$) the operator $T:= \lambda_0+B$ is mixing in the Gaussian sense.
\end{example}
\bpf Apply Corollary \ref{Mullererie} with $f(z):=\lambda_0+z$.
\epf

\begin{remark} If one only assumes that $\ker^*(B)$ is dense in $X$ and $\ker^*(B)\subseteq {\rm Ran}(B)$, then one can conclude that $\lambda_0+B$ is \emph{topologically} mixing for any $\lambda_0\in\T$ (see \cite[Corollary 2.3]{BMbook}). But these weaker assumptions do not entail mixing in the Gaussian sense. For example, if $B$ is a compact weighted backward shift on $c_0$ or $\ell_p$, then $\lambda_0+B$ is not even frequently hypercyclic by \cite[Theorem 1.2]{Shk}, since its spectrum is countable.%For example, if $B$ is a weighted backward shift on $\ell_p$ or $c_0$, %with weights tending to $0$, then it is easily checked that the only eigenvalue of $\lambda_0+B$ is $\lambda_0$; and this implies that $\lambda_0+B$ does not admit a Gaussian invariant measure with full support (see \cite[Theorem 5.46]{BMbook}).  
\end{remark}

\section{Some further comments and questions}\label{Section 6}
We collect in this last section some natural questions which arise in connection with this work.
\par\medskip
To begin with, observe that since the space  \(c_{0}:=c_{0}(\Z_{+})\) has separable dual,  the ball \(\mathcal{B}_{1}(c_{0})\) endowed with the topology \sote\ is a Polish space. So it makes sense to ask whether the results from Sections \ref{Section 3} and \ref{Section 4} can be extended to typical contractions of \(c_{0}\).

\begin{question}\label{Q1}
 Are Theorem \ref{Th1} and Proposition \ref{trueProp2} still valid for $X=c_0$?
\end{question}

Observe that the proofs of Theorem \ref{Th1} and Proposition \ref{trueProp2} use in a crucial way the weak compactness of closed balls of \(\lp\), \(1<p<\infty\).
\par\medskip
We are able to prove the \sote-\,density of \(\mathcal{N}_{1}(X)\) in \(\bbx\) only in the case where \(X=\ell_{2}\), and, as a consequence, Corollaries \ref{smallspectrum} and \ref{Autre corollaire} as well as Propositions \ref{powerreg} and \ref{Prop 4.14} are proved only in the Hilbertian setting. It would thus be interesting to be able to answer the following:

\begin{question}\label{Q 2}
 Let \(X=\lp\), \(1<p\neq 2<\infty\). Is it true that \(\mathcal{N}_{1}(X)\) is \sote-\,dense in \(\bbx\)?
\end{question}

Corollary \ref{smallspectrum} and its proof motivate the next question:

\begin{question}\label{Q 3}
 Let \(X=\ell_{2}\) (or even \(X=\lp\), \(1<p<\infty\)). Let \(a\in\T\). Is it true that the set of all \(T\in\bbx\) such that \(\sigma (T)=\{a\}\) is \sote-\,dense in \(X\)?
\end{question}

In relation to this question, it is shown in \cite[Proposition 3.10]{GMM2} that if \(X=\lp\) for some \(1\le p<\infty\), the set of all \(T\in\bbx\) such that \(\sigma (T)\subseteq \T\) is \sot-\,dense in \(\bbx\). The proof given there shows that for \(1<p<\infty\), this set of operators is \sote-\,dense in \(\bbx\) as well.
\par\medskip
The results of this note show that essentially no known criterion of a spectral flavor can be brought to use to show that a typical \(T\in\mathcal{B}_{1}(\lp)\), \(1<p\neq 2<\infty\), has a non-trivial invariant subspace. In this line of thought, let us mention here an open question from \cite{GMM2}. An important criterion for showing the existence of a non-trivial invariant subspace for operators is given by the famous Lomonosov Theorem, which states that if the commutant \(\{T\}'\) of an operator \(T\) on a Banach space \(X\) contains an operator which is not a multiple of the identity operator and which commutes with a non-zero compact operator, then \(T\) has a non-trivial invariant subspace. It is shown in \cite[Theorem 7.5]{GMM2} that a typical $T\in\mathcal B_1(\ell_2)$ does \emph{not} commute with any non-zero compact operator.

\begin{question}\label{Q 4}
 Let \(X=\lp\), \(1<p<\infty\). Is it true that a typical \(T\in\bbx\) does not satisfy the assumptions of the Lomonosov Theorem? Is it at least true that a typical \(T\in\bbx\) does not commute with a non-zero compact operator?
\end{question}

Lastly, we mention that if \(p>2\), Theorem \ref{Th1} and Proposition \ref{trueProp2}  hold true also for a typical \(T\in(\mathcal{B}_{1}(\lp),\sot)\). Indeed, by \cite[Theorem 5.12]{GMM2}, any \sote-\,comeager subset of \(\mathcal{B}_{1}(\lp)\) is also \sot\,-\,comeager.

\begin{question}\label{Q 5}
 Let \(1<p<2\). Do the statements of Theorem \ref{Th1} and Proposition \ref{trueProp2}  hold true also for a typical \(T\in(\mathcal{B}_{1}(\lp),\sot)\)?
\end{question}

Observe that the answer to Question \ref{Q 5} is negative if \(p=1\) or \(p=2\). Indeed, it is proved in \cite{GMM2} (resp. in \cite{EM}) that a typical \(T\in(\mathcal{B}_{1}(\ell_{1}),\sot)\) (resp. a typical \(T\in(\mathcal{B}_{1}(\ell_{2}),\sot)\)) is such that any \(\lambda \in\D\) is an eigenvalue of \(T\), and that \(T^{*}\) is a surjective isometry. So \(T^{*}\) has property \((\beta )\) (see \cite[Proposition 1.6.6]{LN}), and hence \(T\) has property \((\delta )\).

\end{document}